\documentclass[11pt]{amsart}

\usepackage{amsmath,amssymb,amsfonts}
\usepackage[colorlinks=true,linkcolor=blue,citecolor=blue,urlcolor=blue]{hyperref}

\newtheorem{thm}{Theorem}[section]

\newtheorem{lemma}[thm]{Lemma}
\theoremstyle{definition}
\newtheorem{defn}[thm]{Definition}
\theoremstyle{remark}
\newtheorem{rem}[thm]{Remark}
\numberwithin{equation}{section}

\begin{document}

\title{Integrality of Averages of Roots of Unity and Perfect Isometries}
\thanks{Accepted for publication in the \emph{Bulletin of the Australian Mathematical Society}.}

\author{Chatchawan Panraksa}
\address{Science Division, Mahidol University International College\newline
Nakhon Pathom, Thailand}
\email{chatchawan.pan@mahidol.ac.th}

\author{Pornrat Ruengrot}
\address{Science Division, Mahidol University International College\newline
Nakhon Pathom, Thailand}
\email{pornrat.rue@mahidol.ac.th}

\date{}

\begin{abstract}

We establish a criterion for the integrality of averages of roots of unity and apply it to settle a conjecture regarding the linearity of functions on $\mathbb{Z}_n$. Specifically, we prove that for any modulus $n \ge 1$, if a function $f: \mathbb{Z}_n \to \mathbb{Z}_n$ satisfies that the averages $\frac{1}{n} \sum_{x=0}^{n-1} \omega^{f(x)+bx}$ (where $\omega=e^{2\pi i/n}$) are algebraic integers for all $b \in \mathbb{Z}_n$, then $f$ is necessarily linear modulo $n$. This provides a short, elementary proof that works uniformly for all $n$ and avoids the finite-field machinery used in previous partial results. Furthermore, when $n=p^r$, we utilize a local-global integrality argument to show that any normalized sum of $p^r$-th roots of unity that is $p$-adically integral must be either $0$ or a single root of unity. As an application, we completely characterize the perfect isometries of the cyclic group $C_{p^r}$: they are precisely those induced by affine permutations $x \mapsto \alpha x + \beta$ with $\gcd(\alpha, p^r)=1$.

\end{abstract}


\maketitle

\section{Introduction}
Let $n\ge1$, write $\mathbb{Z}_n=\{0,1,\dots,n-1\}$, and fix a primitive $n$th root of unity $\omega=e^{2\pi i/n}$. For a function $f:\mathbb{Z}_n\longrightarrow\mathbb{Z}_n$ and a parameter $b\in\mathbb{Z}_n$, we consider the average
\begin{equation}\label{eq:muab}
	\mu_{b}(f)\;:=\;\frac1n\sum_{x=0}^{n-1}\omega^{\,f(x)+bx}.
\end{equation}
Linear functions $x\mapsto \alpha x+\beta$ clearly make $\mu_{b}(f)$ either $0$ or an $n$th root of unity, hence an algebraic integer. It was shown in \cite{panraksa2017note} that when $n$ is prime, if $\mu_{b}(f)$ is an algebraic integer for all $b\in\mathbb{Z}_n$, then $f$ is representable by a linear polynomial modulo $n$. The proof there used the theory of permutation polynomials over finite fields and a result of Stothers \cite{stothers1990permutation}.

In this paper, we give a short, elementary argument that works for all $n$ without using finite-field machinery and we use it to obtain consequences in two directions. First, for every modulus $n\ge1$, we prove that if the averages $\mu_b(f)$ are algebraic integers for all $b\in\mathbb Z_n$, then $f$ must be linear modulo $n$ (Theorem~\ref{thm:alln}). This settles the conjecture posed in \cite{panraksa2017note} and gives a short, elementary proof that avoids permutation-polynomial machinery, working uniformly for all $n$.

Second, specializing to $n=p^r$ and using a local--global integrality argument at the unique prime above $p$, we show that any normalized sum of $p^r$-th roots of unity that is $p$-adically integral must be either $0$ or a single $p^r$-th root of unity (Theorem~\ref{thm:allpr}). As an application, we completely determine the perfect isometries of the cyclic group $C_{p^r}$. They are precisely those induced by affine permutations $x\mapsto \alpha x+\beta$ with $\gcd(\alpha,p^r)=1$ (Theorem~\ref{thm:perfect}).

These results illustrate how a basic fact about roots-of-unity averages, combined with a minimal amount of cyclotomic number theory, yields clean structural theorems for functions on $\mathbb Z_n$ and for character-theoretic isometries.

\section{Averages of Roots of Unity: A Criterion}\label{sec:criterion}
We record the well-known and very useful criterion for when an average of roots of unity is an algebraic integer; see \cite[Lemma 4]{panraksa2017note}.

\begin{lemma}\label{lem:criterion}
	Let $\omega_1,\ldots,\omega_n$ be complex roots of unity and put
	\[
		\mu=\frac1n\sum_{j=1}^n\omega_j.
	\]
	Then $\mu$ is an algebraic integer if and only if either $\sum_{j=1}^n\omega_j=0$ or $\omega_1=\cdots=\omega_n$.
\end{lemma}

\begin{proof}
	The ``if'' direction is immediate. For the converse, if the $\omega_j$ are not all equal then the triangle inequality gives $|\mu|<1$, and similarly each Galois conjugate $\mu'$ satisfies $|\mu'|\le1$ with at least one strict inequality. The product of all conjugates of $\mu$ is an algebraic integer of absolute value $<1$, hence $0$, so $\mu$ itself must be $0$.
\end{proof}

\section{Main result}\label{sec:main}
The main result of this note is the following theorem. This settles \cite[Conjecture 9]{panraksa2017note}.
\begin{thm}\label{thm:alln}
	Let $n\ge 1$, $\omega=e^{2\pi i/n}$, and $f:\mathbb Z_n\to\mathbb Z_n$.
	If for every $b\in\mathbb Z_n$ the average
	\[
		\mu_{b}(f)=\frac1n\sum_{x=0}^{n-1}\omega^{\,f(x)+bx}
	\]
	is an algebraic integer, then there exist $\alpha, \beta \in \mathbb{Z}_n$ such that
	$f(x)\equiv \alpha x+\beta \pmod n$ for all $x\in\mathbb Z_n$.
\end{thm}

\begin{proof}
	For each $b\in\mathbb{Z}_n$, Lemma~\ref{lem:criterion} applied to the multiset $\{\omega^{\,f(x)+bx}:x\in\mathbb{Z}_n\}$ says that: either the sum $\sum_x\omega^{\,f(x)+bx}$ vanishes or all terms are equal, i.e.\ $f(x)+bx$ is constant modulo $n$.

	Suppose that \emph{no} $b$ makes $f(x)+bx$ constant modulo $n$. Then for every $b$,
	\[
		\sum_{x=0}^{n-1}\omega^{\,f(x)+bx}=0.
	\]
	Summing these equalities over all $b\in\mathbb{Z}_n$ and reversing the order of summation gives
	\[
		\sum_{b=0}^{n-1}\sum_{x=0}^{n-1}\omega^{\,f(x)+bx}
		=\sum_{x=0}^{n-1}\omega^{\,f(x)}\!\sum_{b=0}^{n-1}\omega^{\,bx}
		= n\,\omega^{\,f(0)}\neq 0,
	\]
	which is a contradiction. Hence, there exists $b_0$ and $c$ with $f(x)+b_0x\equiv c$ for all $x$. Thus, $f(x)\equiv \alpha x+\beta \pmod n$ as required.
\end{proof}

\begin{rem}[Comparison with the prime case]
	In the prime case $n=p$, the proof in \cite[Theorem 7]{panraksa2017note} deduced linearity by showing that for every $b$ either $f(x)+bx$ is constant or the map $x\mapsto f(x)+bx$ is a permutation, and then invoking a result of Stothers on permutation polynomials to force $\deg f\le1$. The argument above avoids that detour entirely and works for all $n$ as well. 
\end{rem}

\section{The case \texorpdfstring{$n=p^r$}{n=p to the r} and the localizations}
\label{sec:localization}

In this section we recast the integrality condition $\,\mu_b(f)\in \mathcal{O}_F\,$ in $p$-adic
terms when $n=p^r$. Set
\[
	\omega=\zeta_{p^r}=e^{2\pi i/p^r},\qquad F=\mathbb{Q}(\omega),\qquad \mathcal{O}_F=\mathbb{Z}[\omega].
\]
For a nonzero prime ideal $\mathfrak{q}\subset \mathcal{O}_F$, the \emph{localization} at $\mathfrak{q}$ is
\[
	(\mathcal{O}_F)_{\mathfrak{q}} \;=\; \left\{\frac{a}{b}\;:\; a,b\in \mathcal{O}_F,\; b\notin \mathfrak{q}\right\} = \{x\in F: v_{\mathfrak{q}}(x)\ge 0\},
\]
where $v_{\mathfrak{q}}$ is the discrete valuation associated to $\mathfrak{q}$. We will
use that $\mathcal{O}_F$ is a Dedekind domain and satisfies the local--global identity
\begin{equation}\label{eq:lg}
	\mathcal{O}_F=\bigcap_{\mathfrak{q}\ne 0}(\mathcal{O}_F)_{\mathfrak{q}}\subset F,
\end{equation}
with the intersection taken over all nonzero prime ideals $\mathfrak{q}\subset \mathcal{O}_F$.

\begin{lemma}\label{lem:total-ram}
	In $F=\mathbb{Q}(\zeta_{p^r})$ the rational prime $p$ is totally ramified. More precisely,
	\[
		(p)=\bigl(1-\omega\bigr)^{\varphi(p^r)}\qquad\bigl(\,\varphi(p^r)=p^{\,r-1}(p-1)\,\bigr).
	\]
	In particular, there is a unique prime ideal $\mathfrak{p}$ of $\mathcal{O}_F$ above $(p)$, namely
	$\mathfrak{p}=(1-\omega)$, and $v_{\mathfrak{p}}(p)=\varphi(p^r)$. Moreover,
	\[
		F_{\mathfrak{p}}\cong \mathbb{Q}_p(\omega),\qquad \mathcal{O}_{F_{\mathfrak{p}}}\cong \mathbb{Z}_p[\omega],
	\]
	and $1-\omega$ is a uniformizer of $\mathcal{O}_{F_{\mathfrak{p}}}$.
\end{lemma}

\noindent(These facts are standard in cyclotomic field theory; we record them to fix notation.)

\medskip

We now relate $p$-adic integrality at $\mathfrak{p}$ to global integrality in $\mathcal{O}_F$.

\begin{thm}\label{thm:local-global}
	Let $p$ be a prime and $r\ge 1$. With notation as above, let $S\in \mathcal{O}_F$ and set $\mu=S/p^r\in F$.
	Let $\mathfrak{p}=(1-\omega)$ be the unique prime of $\mathcal{O}_F$ above $(p)$. If (under the natural embedding
	$F\hookrightarrow F_{\mathfrak{p}}\cong \mathbb{Q}_p(\omega)$) we have
	\[
		\mu\in \mathcal{O}_{F_{\mathfrak{p}}}\cong \mathbb{Z}_p[\omega],
	\]
	then $\mu\in \mathcal{O}_F$. In particular, $\mu$ is an algebraic integer of $F$.
\end{thm}

\begin{proof}
	By \eqref{eq:lg}, it suffices to show $\mu\in (\mathcal{O}_F)_{\mathfrak{q}}$ for every nonzero prime
	$\mathfrak{q}\subset \mathcal{O}_F$.

	\emph{Case 1: $\mathfrak{q}\nmid \mathfrak{p}$.} In the discrete valuation ring $(\mathcal{O}_F)_{\mathfrak{q}}$ we have $v_{\mathfrak{q}}(p)=0$,
	so $p^r$ is a unit. Since $S\in \mathcal{O}_F\subset (\mathcal{O}_F)_{\mathfrak{q}}$, it follows that
	\[
		v_{\mathfrak{q}}(\mu)=v_{\mathfrak{q}}(S)-v_{\mathfrak{q}}(p^r)\ge 0-0=0,
	\]
	hence $\mu\in (\mathcal{O}_F)_{\mathfrak{q}}$.

	\emph{Case 2: $\mathfrak{q}=\mathfrak{p}$.} By hypothesis, $\mu\in \mathcal{O}_{F_{\mathfrak{p}}}$. Since
	$(\mathcal{O}_F)_{\mathfrak{p}}=\mathcal{O}_{F_{\mathfrak{p}}}\cap F$ (as subrings of $F_{\mathfrak{p}}$), we conclude
	$\mu\in (\mathcal{O}_F)_{\mathfrak{p}}$.

	Therefore, $\mu\in (\mathcal{O}_F)_{\mathfrak{q}}$ for all nonzero primes $\mathfrak{q}$, and \eqref{eq:lg} yields
	$\mu\in \mathcal{O}_F$.
\end{proof}

Note that the form of $S$ is irrelevant for Theorem~\ref{thm:local-global}; only $S\in\mathcal{O}_F$ matters. When $S$ is a sum of roots of unity, we obtain the following result.

\begin{thm}\label{thm:allpr}
	Let $p$ be a prime and $r\ge 1$. Write $\omega=\zeta_{p^r}$, $F=\mathbb{Q}(\omega)$, $\mathcal{O}_F=\mathbb{Z}[\omega]$.
	Let
	\[
		S=\sum_{j=0}^{p^r-1} m_j\,\omega^j,\qquad m_j\in\mathbb{Z}_{\ge 0},\qquad \sum_{j=0}^{p^r-1} m_j=p^r.
	\]
	If $\dfrac{S}{p^r}\in \mathcal{O}_{F_{\mathfrak{p}}}\cong \mathbb{Z}_p[\omega]$, then either $S=0$ or
	$S=p^r\,\omega^{j_0}$ for some $0\le j_0<p^r$.
\end{thm}

\begin{proof}
	By Theorem~\ref{thm:local-global}, the hypothesis implies $S/p^r\in \mathcal{O}_F$. Applying the roots-of-unity average criterion (Lemma~\ref{lem:criterion}) to the multiset of terms, we see that the average $S/p^r$ is an algebraic integer only in the two cases stated.
\end{proof}

\section{Application to Perfect Isometries of Cyclic Groups of Prime Power Order}

We now apply the integrality criterion from the previous section to a problem in
character theory: determining the \emph{perfect isometries} of a cyclic group of order $p^r$.

Let $n=p^r$ and consider the cyclic group $G = C_n = \langle g \rangle$.
The $n$ irreducible complex characters of $G$ are
\[
	\mathrm{Irr}(G) = \{\chi_0,\chi_1,\ldots,\chi_{n-1}\}, \qquad
	\chi_j(g^a) = \omega^{ja},\ \ \omega = e^{2\pi i/n}.
\]

Brou\'e's original definition of a perfect isometry~\cite{broue1990isometries}
is given with respect to a $p$-modular system $(K,\mathcal{O},k)$
in which $\mathcal{O}$ is a complete discrete valuation ring whose maximal
ideal contains $p$, the field of fractions $K$ has characteristic~0 and
contains all character values of~$G$, and $k$ is the residue field of
characteristic~$p$.
For $G=C_{p^r}$ all character values lie in $F=\mathbb{Q}(\omega)$
with $\omega=\zeta_{p^r}$, and $p$ is totally ramified in~$F$.
On the $p$-adic side we have $K\cong\mathbb{Q}_p(\omega)$ and
$\mathcal{O}\cong\mathbb{Z}_p[\omega]$, a complete DVR with uniformizer
$1-\omega$ and residue field $\mathbb{F}_p$.
Since $G$ is abelian, $|C_G(g)|=|G|=p^r$ for every~$g$.
Hence Brou\'e's integrality condition
\[
	\frac{\mu(g,h)}{|C_G(g)|},\;
	\frac{\mu(g,h)}{|C_G(h)|}\in\mathcal{O}
\]
reduces precisely to
\[
	\frac{\mu(g,h)}{p^r}\in\mathbb{Z}_p[\omega]\quad\text{for all }g,h\in G.
\]
We therefore adopt $\mathcal{O}=\mathbb{Z}_p[\omega]$ throughout.

\begin{defn}[Perfect isometry for $C_{p^r}$]
	Let $n=p^r$, $G=C_n=\langle g\rangle$, and $\omega=e^{2\pi i/n}$.
	A bijection $I:\mathrm{Irr}(G)\to\mathrm{Irr}(G)$ is a
	\emph{perfect isometry} if the associated function
	\[
		\mu_I(g^a,g^b)=\sum_{j=0}^{n-1} I(\chi_j)(g^a)\,\chi_j(g^b)
	\]
	satisfies the following conditions:
	\begin{itemize}
		\item[(i)] (\emph{integrality}) $\displaystyle \frac{\mu_I(g^a,g^b)}{n}\in
			      \mathbb{Z}_p[\omega]$ for all $a,b\in\mathbb{Z}_n$;
		\item[(ii)] (\emph{separation}) if $\mu_I(g^a,g^b)\neq 0$, then $g^a$ and
		      $g^b$ are either both the identity or both non-identity elements.
	\end{itemize}
\end{defn}

\begin{rem}
	Intuitively, condition~(i) asserts that the normalized sums appearing in $\mu_I$
	are \emph{$p$-adically integral}, while condition~(ii) ensures that the nonzero values of
	$\mu_I$ respect the partition of $G$ into $\{1\}$ and $G\setminus\{1\}$.
\end{rem}

Every such bijection $I$ is determined by a permutation $f$ of $\mathbb{Z}_n$
such that $I(\chi_j) = \chi_{f(j)}$.  In this notation,
\[
	\mu_I(g^a,g^b) = \sum_{j=0}^{n-1} \omega^{a f(j) + bj}.
\]

We can now describe all permutations $f$ giving rise to perfect isometries.

\begin{thm}\label{thm:perfect}
	Let $n=p^r$.  A bijection $f : \mathbb{Z}_n \to \mathbb{Z}_n$ gives a perfect
	isometry of the group $C_n$ if and only if there exist $\alpha,\beta \in \mathbb{Z}_n$
	with $\gcd(\alpha,n)=1$ such that
	\[
		f(x) \equiv \alpha x + \beta \pmod n \quad \text{for all } x\in\mathbb{Z}_n.
	\]
\end{thm}

\begin{proof}
	$(\Rightarrow)$\ Suppose $f$ yields a perfect isometry.  By condition~(i),
	for every $a,b\in\mathbb{Z}_n$,
	\[
		\frac{1}{n}\sum_{j=0}^{n-1}\omega^{a f(j)+bj} \in \mathbb{Z}_p[\omega].
	\]
	Fix $a=1$.  Then for all $b$,
	\[
		\frac{1}{n}\sum_{j=0}^{n-1}\omega^{f(j)+bj} \in \mathbb{Z}_p[\omega].
	\]
	By Theorem~\ref{thm:local-global} this average is an algebraic integer, so Theorem~\ref{thm:alln} applies and gives $\alpha,\beta\in\mathbb{Z}_n$ with
	\[
		f(j)\equiv \alpha j + \beta \pmod n \quad\text{for all }j.
	\]
	Since $f$ is bijective, multiplication by $\alpha$ must be invertible modulo~$n$;
	hence $\gcd(\alpha,n)=1$.

	$(\Leftarrow)$\ Conversely, assume $f(x)\equiv \alpha x+\beta$ with
	$\gcd(\alpha,n)=1$.  Then
	\[
		\mu_I(g^a,g^b)
		= \sum_{j=0}^{n-1}\omega^{a(\alpha j+\beta)+bj}
		= \omega^{a\beta}\sum_{j=0}^{n-1}\omega^{(a\alpha+b)j}.
	\]
	The inner sum equals $n$ if $a\alpha+b\equiv0\pmod n$ and $0$ otherwise, so
	\[
		\frac{\mu_I(g^a,g^b)}{n}
		= \begin{cases}
			\omega^{a\beta}\in \mathbb{Z}_p[\omega], & b\equiv -a\alpha \pmod n, \\[4pt]
			0,                                       & \text{otherwise}.
		\end{cases}
	\]
	Thus condition~(i) holds.  For~(ii), if $\mu_I(g^a,g^b)\ne0$ then
	$b\equiv -a\alpha\pmod n$, and because $\alpha$ is a unit, we have
	$a=0$ if and only if $b=0$.  Hence $g^a$ and $g^b$ are simultaneously
	the identity or both non-identity elements, satisfying the separation property.
\end{proof}

\section*{Acknowledgement}
The authors thank the anonymous referee for valuable comments.

\end{document}